\pdfoutput=1
\documentclass[11pt,a4paper]{amsart}
\usepackage[margin=1in]{geometry}
\usepackage[utf8]{inputenc}
\usepackage[T1]{fontenc}
\usepackage[dvipsnames]{xcolor}
\usepackage{microtype}
\usepackage{fnpct}

\usepackage{amsmath}
\usepackage{amssymb}
\usepackage[mathscr]{eucal}
\usepackage{tikz-cd}

\renewcommand{\projlim}{\varprojlim}

\usepackage{enumitem}
\usepackage{booktabs}
\usepackage[pdfusetitle,colorlinks]{hyperref}
\hypersetup{bookmarksdepth=2,pdfencoding=unicode,allcolors=MidnightBlue}
\usepackage[capitalise,noabbrev,nosort]{cleveref}

\crefname{equation}{}{}
\crefformat{enumi}{(#2#1#3)}
\crefname{enumi}{}{}

\numberwithin{equation}{section}
\theoremstyle{plain}
\newtheorem{Theorem}{Theorem}

\crefname{Theorem}{Theorem}{Theorems}
\newtheorem{theorem}[equation]{Theorem}
\newtheorem{proposition}[equation]{Proposition}
\newtheorem{lemma}[equation]{Lemma}
\newtheorem{corollary}[equation]{Corollary}

\theoremstyle{definition}
\newtheorem{definition}[equation]{Definition}

\newtheorem{example}[equation]{Example}
\newtheorem{question}[equation]{Question}

\theoremstyle{remark}
\newtheorem{remark}[equation]{Remark}

\let\oldSS\SS\let\SS\relax

\newcommand{\ZZ}{\mathbf{Z}}
\newcommand{\QQ}{\mathbf{Q}}

\newcommand{\SS}{\mathbf{S}}

\newcommand{\E}{\mathrm{E}}

\newcommand{\CAlg}{\operatorname{CAlg}}
\newcommand{\D}{\operatorname{D}}
\newcommand{\Fun}{\operatorname{Fun}}

\newcommand{\QCoh}{\operatorname{QCoh}}
\newcommand{\Shv}{\operatorname{Shv}}
\newcommand{\Sm}{\operatorname{Sm}}
\newcommand{\cofib}{\operatorname{cofib}}
\newcommand{\coker}{\operatorname{coker}}

\newcommand{\hyp}{\operatorname{hyp}}
\newcommand{\id}{\operatorname{id}}
\newcommand{\op}{\operatorname{op}}

\newcommand{\X}{\mathord{-}}
\newcommand{\unit}{\mathbf{1}}
\newcommand{\sd}{\mathbin{\triangle}}

\newcommand{\cat}[1]{\mathscr{#1}}
\newcommand{\Cat}[1]{\mathsf{#1}}
\newcommand{\shf}[1]{\mathcal{#1}}

\title{The smashing spectrum of sheaves}
\author{Ko Aoki}
\address{Max Planck Institute for Mathematics,
  Vivatsgasse 7, 53111 Bonn, Germany
}
\email{aoki@mpim-bonn.mpg.de}
\date{\today}

\begin{document}

\begin{abstract}
For an arbitrary \(\infty\)-topos,
  we classify the smashing localizations
  in the \(\infty\)-category
  of sheaves
  valued in derived vector spaces:
  Any of them is
  the restriction functor to a (unique) closed subtopos.
  Our proof is based on the existence of a Boolean cover.

  This result
  in particular gives us the first example
  of a nonzero presentably symmetric monoidal stable \(\infty\)-category
  whose smashing spectrum has no points.

  Combining this with
  the sheaves-spectrum adjunction,
  we obtain a Tannaka-type
  categorical reconstruction result for 
  locales.
\end{abstract}

\maketitle

\section{Introduction}\label{s:intro}

Given a presentably symmetric monoidal stable \(\infty\)-category~\(\cat{C}\),
understanding its localizations
is fundamental.
Particularly,
\emph{smashing localizations},
i.e.,
localizations \(L\colon\cat{C}\to\cat{C}\) that are equivalent to
\(L(\unit)\otimes\X\),
are classical and important in stable homotopy theory;
cf.~\cite[Section~1]{Ravenel84}.
A basic observation
is that
they form a \emph{frame},
i.e., a complete lattice
satisfying the distributivity law
\(
\bigvee_{i\in I}x\wedge y_i
=
x\wedge\bigvee_{i\in I}y_i
\) for any \(x\) and \((y_i)_{i\in I}\).
The associated locale,
for which we write \(\Sm(\cat{C})\),
is called
the \emph{smashing spectrum} of~\(\cat{C}\).
See~\cite{ttg-sm} and references therein
about this construction.

Computing \(\Sm(\cat{C})\)
for given~\(\cat{C}\)
is typically extremely difficult;
e.g.,
even the most standard case \(\cat{C}=\Cat{Sp}\),
the \(\infty\)-category of spectra,
contains the telescope conjecture.
Although the conjecture itself
was recently settled
in~\cite{BHIS},
our knowledge of \(\Sm(\Cat{Sp})\)
remains very limited.
The case \(\cat{C}=\D(R)\)
for a commutative ring~\(R\)
has been actively studied:
When \(R\) is noetherian,
\(\Sm(\D(R))\)
is the Zariski spectrum of~\(R\);
see~\cite[Section~3]{Neeman92}
or~\cite[Section~6]{BalmerFavi11}.
In general,
there are many more smashing localizations;
see, e.g.,
the calculation
when \(R\) is a valuation ring
in~\cite[Section~5]{BazzoniStovicek17}.
Still,
determining \(\Sm(\D(R))\)
is unapproachable in general.

In this paper,
as~\(\cat{C}\)
we consider
\(\Shv(\cat{X};\D(k))\)
for an \(\infty\)-topos~\(\cat{X}\)
and a field~\(k\).
For example,
for a topological space (or a locale)~\(X\),
by considering
\(\cat{X}=\Shv(X)^{\hyp}\),
we get the derived category of sheaves on~\(X\)
valued in classical \(k\)-vector spaces as~\(\cat{C}\).
There is an obvious class of smashing localizations on~\(\cat{C}\):
For any closed subtopos~\(\cat{Z}\),
the composite
\begin{equation*}
  \Shv(\cat{X};\D(k))
  \xrightarrow{i^*}
  \Shv(\cat{Z};\D(k))
  \xrightarrow{i_*}
  \Shv(\cat{X};\D(k))
\end{equation*}
is smashing due to the projection formula.
The main theorem of this paper states that 
there are no other smashing localizations:

\begin{Theorem}\label{main}
  For an \(\infty\)-topos \(\cat{X}\) and a field~\(k\),
  the smashing spectrum of~\(\Shv(\cat{X};\D(k))\) is
  canonically isomorphic to
  the locale of subterminal objects.
\end{Theorem}

We note that our proof uses a nontrivial theorem
from (higher) topos theory as an input.

\begin{remark}\label{a521a769a8}
  The phenomenon of
  the smashing spectrum
  only depending on
  the underlying locale
  in \cref{main}
  is special to our choice of coefficients:
  An interesting example is
  given by Clausen--Mathew~\cite[Theorem~1.9]{ClausenMathew21},
  which says
  that for the étale \(\infty\)-topos
  of a nice scheme,
  the hypercompletion
  of \(\Cat{Sp}\)-valued sheaves
  is smashing.
  Typically, hypercompletion is nontrivial
  (even over~\(\mathrm{KU}\)),
  as demonstrated by the argument of Wieland
  in~\cite[Warning~7.2.2.31]{LurieHTT}
\end{remark}

The importance of this theorem is two-fold:
First,
as already explained in~\cite[Section~1.2]{ttg-sm},
we can deduce the following Tannaka-type duality
for locales
by exploiting the main result of~\cite{ttg-sm}:

\begin{Theorem}\label{tloc_k}
  For a field~\(k\),
  the functor
  \begin{equation*}
    \Shv(\X;\D(k))
    \colon
    \Cat{Loc}^{\op}
    \to\CAlg_{\D(k)}(\Cat{Pr})
  \end{equation*}
  is fully faithful.
\end{Theorem}

\begin{remark}\label{c69acd17e5}
  In \cref{tloc_k},
  note that
  the category of locales has a canonical bicategorical structure;
  see \cref{tloc_k_2} for a version
  where this is taken into account.
\end{remark}

\begin{remark}\label{xyy9wg}
  In \cref{tloc_k},
  we cannot replace~\(k\)
  with~\(\ZZ\); i.e.,
  \begin{equation}
    \label{e:wyoo2}
    \Shv(\X;\D(\ZZ))
    \colon
    \Cat{Loc}^{\op}
    \to\CAlg_{\D(\ZZ)}(\Cat{Pr})
  \end{equation}
  is not fully faithful:
  Consider the Sierpiński space, i.e., the Zariski spectrum of
  a discrete valuation ring.
  Then its sheaf \(\infty\)-category
  is equivalent to \(\Fun(\Delta^{1},\D(\ZZ))\) with
  the pointwise symmetric monoidal structure.
  There is a symmetric monoidal functor
  \(\Fun(\Delta^{1},\D(\ZZ))\to\D(\ZZ)\)
  that pointwise carries~\(F\)
  to the limit of \(\QQ\otimes F(0)\to\QQ\otimes F(1)\gets F(1)\),
  but it does not come from any locale map.

  However,
  we can show that
  \begin{equation*}
    \Shv(\X;\D(\ZZ)_{\geq0})
    \colon
    \Cat{Loc}^{\op}
    \to\CAlg_{\D(\ZZ)_{\geq0}}(\Cat{Pr})
  \end{equation*}
  is fully faithful
  by a similar proof to that of \cref{tloc_k}:
  For a locale~\(X\)
  we can identify
  the (unstable) smashing spectrum of \(\Shv(X;\D(\ZZ)_{\geq0})\)
  with~\(X\)
  by applying \cref{main} for prime fields.
  Then the desired result follows from considering \(\cat{V}=\D(\ZZ)_{\geq0}\)
  in the proof of \cref{tloc_k}.

  In future work,
  we will prove that \cref{e:wyoo2} is fully faithful when
  restricting to compact Hausdorff spaces.
  Our proof uses a certain ``continuous'' version of~\(\Sm\).
\end{remark}

Second,
\cref{main}
has the following consequence:

\begin{Theorem}\label{ptless}
  Any locale is the smashing spectrum
  of some presentably symmetric monoidal stable \(\infty\)-category.
  In particular,
  there is a nonzero presentably symmetric monoidal stable \(\infty\)-category
  whose smashing spectrum has no points.
\end{Theorem}

Recall that
a locale is called \emph{spatial}
if it is isomorphic to
the frame of open subsets of a (unique sober) topological space.
We can ask the following:

\begin{question}\label{9ee60e159f}
  Under what condition on
  a presentably symmetric monoidal stable \(\infty\)-category~\(\cat{C}\),
  is the smashing spectrum \(\Sm(\cat{C})\)
  spatial?
\end{question}

This question is interesting
because if we know that \(\Sm(\cat{C})\) is spatial,
we can study smashing localizations of~\(\cat{C}\)
by studying the points
and topology of \(\Sm(\cat{C})\)
separately.
In the compactly generated rigid case,
people have attempted to prove the spatiality,
but it is still open;
see~\cite[Appendix~A]{BalchinStevenson}.
Lastly,
we note that even the answer to the following
is not known:

\begin{question}\label{b36b7f6b5d}
  Is \(\Sm(\D(R))\) always spatial
  for a commutative ring~\(R\)?
\end{question}

This paper is organized as follows:
We first recall facts
on sheaves on Boolean locales
in \cref{s:b_loc}.
Using those,
we prove \cref{main}
and deduce
(a refined version of) \cref{tloc_k}
in \cref{s:ex}.

\subsection*{Acknowledgments}\label{ss:ack}

I thank Alexander Efimov and Peter Scholze
for helpful communications;
the spatial case of \cref{main} was obtained
during a discussion with them.
I also thank Scholze for his useful comments on the draft.
I thank the Max Planck Institute for Mathematics
for its hospitality.

\subsection*{Conventions}\label{ss:conv}

Regarding smashing spectrum,
we continue using the notations from~\cite{ttg-sm}.

We do not use~\(\mathrm{L}\) or~\(\mathrm{R}\)
to signify how functors are derived;
e.g.,
\(\otimes\) on
derived \(\infty\)-categories
means the derived tensor product.

\section{Boolean locales and sheaves thereon}\label{s:b_loc}

In this section,
we study sheaves on Boolean locales.
In \cref{ss:bool}, we review
the notion of Boolean locales.
In \cref{ss:shv_b},
we consider sheaves on them.
In \cref{ss:shv_k},
we focus on the case of \(\D(k)\)-valued sheaves
for a field~\(k\).

\subsection{Boolean locales}\label{ss:bool}

We recall the following definition:

\begin{definition}\label{96e552485b}
  For an element~\(x\) in a distributive lattice,
  its \emph{complement} is an element~\(x'\)
  satisfying
  \(x\wedge x'=\bot\) and \(x\vee x'=\top\).
  By distributivity, it is necessarily unique
  as
  \begin{equation*}
    x'
    =x'\wedge\top
    =x'\wedge(x\vee x'')
    =(x'\wedge x)\vee(x'\wedge x'')
    =\bot\vee(x'\wedge x'')
    =x'\wedge x''
  \end{equation*}
  and similarly \(x''=x'\wedge x''\)
  hold
  when \(x'\) and~\(x''\) are complements of~\(x\).
  A \emph{Boolean lattice} is
  a distributive lattice in which
  every element has a complement.
  A locale is called \emph{Boolean}
  if its underlying frame is Boolean.
  In a Boolean lattice,
  we write \(x\setminus y\)
  for \(x\wedge y'\)
  where \(y'\) is the complement of~\(y\).
\end{definition}

\begin{example}\label{6981a75aa0}
  In the poset of open subsets of a topological space,
  complementable elements correspond to clopen subsets.
  Therefore, a spatial locale is Boolean if and only if it is discrete.
  In other words,
  a Boolean locale is either discrete or nonspatial.
\end{example}

\begin{remark}\label{58f5b2c0fe}
  A \emph{complete Boolean algebra}
  is another name for a Boolean locale (or frame).
\end{remark}

We here give two typical nondiscrete examples:

\begin{example}\label{random}
  Consider
  the poset of Borel subsets
  (or Lebesgue-measurable subsets)
  of the interval \([0,1]\).
  We take a quotient
  with respect to
  the equivalence relation
  identifying~\(B\) and~\(B'\)
  when their symmetric difference \(B\sd B'\)
  has measure zero.
  Then we get a Boolean frame,
  which is commonly called the \emph{random algebra}
  by set theorists (cf.~\cite{Solovay70}).
  This does not have any points.
\end{example}

\begin{example}\label{cohen}
  We again consider the poset of Borel subsets
  (or open subsets) of~\([0,1]\),
  but this time
  we consider the equivalence relation
  that identifies~\(B\) and~\(B'\)
  when \(B\sd B'\) is a meager set,
  i.e., a countable union
  of nowhere dense subsets.
  The quotient
  is again a Boolean frame,
  which is often called the \emph{Cohen algebra}
  because of its role in forcing.
  This does not have any points.
\end{example}

\begin{remark}\label{fd636e88b0}
  \Cref{random,cohen} are classical;
  most famously,
  they appeared in von~Neumann's 1936--1937 lectures
  on continuous geometry~\cite[Part~III]{vonNeumann60},
  where
  he showed that these two locales are not isomorphic.
\end{remark}

\begin{remark}\label{87d9356c3b}
  If a locale does not have any points,
  neither does its sheaf \(1\)-topos.
  The example of a nontrivial \(1\)-topos
  without points
  given in~\cite[IV.7.4]{SGA4a},
  which is attributed to Deligne there,
  is based on this observation;
  it is simply the sheaf \(1\)-topos of the locale given in \cref{random}.
\end{remark}

\subsection{Sheaves on a Boolean locale}\label{ss:shv_b}

We recall the following
from~\cite[Proposition~3.18]{ttg-sm}:

\begin{theorem}\label{shv_loc}
  Let \(X\) be a locale with frame of opens~\(F\)
  and \(\cat{C}\) an \(\infty\)-category with limits.
  Then a presheaf
  \(\shf{F}\colon F^{\op}\to\cat{C}\),
  is a sheaf if and only if
  it satisfies the following:
  \begin{enumerate}
    \item\label{i:red}
      The value \(\shf{F}(\bot)\) is final.
    \item\label{i:exc}
      The square
      \begin{equation*}
        \begin{tikzcd}
          \shf{F}(V\vee V')\ar[r]\ar[d]&
          \shf{F}(V)\ar[d]\\
          \shf{F}(V')\ar[r]&
          \shf{F}(V\wedge V')
        \end{tikzcd}
      \end{equation*}
      is cartesian
      for any two opens~\(V\) and~\(V'\).
    \item\label{i:fil}
      The morphism
      \begin{equation*}
        \shf{F}\Bigl(\bigvee D\Bigr)
        \to
        \projlim_{U\in D}
        \shf{F}(U)
      \end{equation*}
      is an equivalence
      for any directed subset~\(D\subset F\).
  \end{enumerate}
\end{theorem}

We prove the following variant in the Boolean case:

\begin{theorem}\label{shv_b}
  Let \(X\) be a Boolean locale with frame of opens~\(F\)
  and \(\cat{C}\) an \(\infty\)-category with limits.
  Then a presheaf is a sheaf
  if and only if
  it carries disjoint joins (see \cref{178b2b9645})
  to products.
\end{theorem}

\begin{remark}\label{5ba629d6e2}
  By \cref{shv_b},
  the notion of sheaves
  on a Boolean locale
  makes sense
  when 
  the coefficient \(\infty\)-category
  only has products
  and not necessarily all limits.
\end{remark}

\begin{definition}\label{178b2b9645}
  Let \(D\) be a Boolean lattice.
  We say that
  a family of elements \((a_i)_{i\in I}\)
  is (pairwise) \emph{disjoint} if
  \(a_i\wedge a_j=\bot\)
  for \(i\neq j\).
\end{definition}

\begin{proposition}\label{shv_kani}
  In the statement of \cref{shv_loc},
  when \(X\) is Boolean,
  \cref{i:exc} can be replaced
  by the condition that
  for any disjoint opens \(V\) and~\(V'\),
  the morphism
  \(\shf{F}(V\amalg V')\to\shf{F}(V)\times\shf{F}(V')\)
  is an equivalence.
\end{proposition}

\begin{proof}
  When we consider
  disjoint opens \(V\) and~\(V'\)
  in~\cref{i:exc},
  we exactly get our condition
  by~\cref{i:red}.
  On the other hand,
  if \(\shf{F}\) satisfies
  \cref{i:red} and our condition,
  the diagram in \cref{i:exc}
  is equivalent to
  \begin{equation*}
    \begin{tikzcd}
      \shf{F}(V\setminus V')
      \times
      \shf{F}(V\wedge V')
      \times
      \shf{F}(V'\setminus V)
      \ar[r]\ar[d]&
      \shf{F}(V\setminus V')
      \times
      \shf{F}(V\wedge V')
      \ar[d]\\
      \shf{F}(V\wedge V')
      \times
      \shf{F}(V'\setminus V)
      \ar[r]&
      \shf{F}(V\wedge V')
      \rlap,
    \end{tikzcd}
  \end{equation*}
  which is cartesian.
\end{proof}

\begin{remark}\label{821682feb4}
  The same argument as in \cref{shv_kani}
  shows a similar statement
  for sheaves on a Stone space (aka totally disconnected compact Hausdorff space):
  A presheaf on a Stone space is a sheaf
  if and only if
  it carries directed joins to limits
  and finite disjoint joins of quasicompact opens to products.
\end{remark}

We need the following,
which is the \(\infty\)-categorical
version of~\cite[Corollary~1.7]{AdamekRosicky94}:

\begin{proposition}\label{chain}
  An \(\infty\)-category~\(\cat{C}\)
  has all filtered colimits
  if and only if
  it has colimits indexed by ordinals.

  For an \(\infty\)-category~\(\cat{C}\)
  having small filtered colimits,
  a functor \(F\colon\cat{C}\to\cat{D}\)
  preserves  filtered colimits
  if and only if it preserves
  colimits indexed by ordinals.
\end{proposition}

\begin{remark}\label{213df081ac}
  Whereas \cref{chain} is true,
  beware that it is not the case
  that any directed poset
  has a cofinal map from an ordinal:
  The poset of the finite subsets
  of an uncountable set
  is such an example.
  However, we also note that
  any countable directed poset
  has a cofinal map from~\(\omega\).
\end{remark}

\begin{remark}\label{x4pjh0}
  The same proof shows
  the refinement of \cref{chain}
  with a bound given by a fixed cardinal~\(\kappa\).
  For example,
  an \(\infty\)-category
  has all \(\kappa\)-small filtered colimits
  if and only if
  it has colimits indexed by ordinals
  in~\(\kappa\).
\end{remark}

\begin{proof}
  We prove the first statement
  since the second one follows from
  the same argument.
  The ``only if'' direction is clear.
  We prove the ``if'' direction.
  Let \(P\) be a directed poset
  with cardinality~\(\kappa\).
  We want to show that
  any diagram \(P\to\cat{C}\) has colimits.
  We proceed by induction on~\(\kappa\).
  If \(\kappa\) is finite,
  \(P\) has a final object
  and its image in~\(\cat{C}\) is the colimit.
  Suppose that \(\kappa\) is infinite
  and that the statement holds for smaller cardinals.
  By~\cite[Lemma~1.6]{AdamekRosicky94}\footnote{Beware a small error in the proof given there:
    With their construction, \(I_{0}=\emptyset\) is not directed.
    We can fix this by reindexing.
  },
  we can write \(P\) as an increasing
  of subposets \(P_{\lambda}\) for \(\lambda<\kappa\)
  such that each \(P_{\lambda}\) is directed
  and of cardinality \(<\kappa\).
  By the inductive hypothesis,
  we can take the colimit of each
  restriction \(P_{\lambda}\to\cat{C}\)
  and then take the colimit of the resulting diagram
  indexed by~\(\kappa\), which exists by our assumption.
  This is a colimit of the original diagram
  by~\cite[Corollary~4.2.3.10]{LurieHTT}.
\end{proof}

\begin{proof}[Proof of \cref{shv_b}]
  By \cref{shv_kani},
  it suffices to show that
  for a presheaf \(\shf{F}\colon F^{\op}\to\cat{C}\)
  carrying finite disjoint joins to products,
  \cref{i:fil} of \cref{shv_loc} holds if and only if
  \(\shf{F}\) carries
  arbitrary disjoint joins to products.
  The ``only if'' direction is clear
  since we can write disjoint joins
  as a filtered colimit of finite disjoint unions.
  Hence we prove the ``if'' direction.

  By \cref{chain},
  it suffices to show that
  \(\shf{F}\) carries
  any colimit indexed by an ordinal~\(\alpha\)
  to a limit.
  Let \(U_{\beta}\) be such a diagram with join~\(U\).
  Then we take
  \begin{equation*}
    V_{\beta}
    =U_{\beta}
    \setminus\bigvee_{\gamma<\beta}U_{\gamma}.
  \end{equation*}
  This family is disjoint and has~\(U\) as its join.
  We have a diagram
  \begin{equation*}
    \shf{F}(U)
    \to\projlim_{\beta<\alpha}\shf{F}(U_{\beta})
    \to\prod_{\beta<\alpha}\shf{F}(V_{\beta})
  \end{equation*}
  and want to show that the first map is an equivalence.
  Since we know that the composite is an equivalence by assumption,
  it suffices to show that the second map is an equivalence.
  However,
  since
  \begin{equation*}
    U_{\beta}=\bigvee_{\gamma\leq\beta}U_{\gamma},
  \end{equation*}
  by using our assumption again,
  we can rewrite the second map as
  \begin{equation*}
    \projlim_{\beta<\alpha}
    \prod_{\gamma\leq\beta}\shf{F}(V_{\gamma})
    \to
    \prod_{\beta<\alpha}\shf{F}(V_{\beta}),
  \end{equation*}
  which is an equivalence.
\end{proof}

\Cref{shv_b} is useful
since it says that the sheaf condition
can be checked on homotopy groups.
Here we record some immediate consequences:

\begin{corollary}\label{xw148x}
  For a Boolean locale~\(X\)
  and \(n\geq-2\),
  the \((n+1)\)-connective/\(n\)-truncated factorization system
  (see~\cite[Example~5.2.8.16]{LurieHTT})
  for \(\Shv(X)\)
  can be computed pointwise.
\end{corollary}

The following
were also obtained in~\cite[Section~A.4.1]{LurieSAG}:

\begin{corollary}\label{post}
  For a Boolean locale~\(X\),
  the \(\infty\)-topos~\(\Shv(X)\)
  is Postnikov complete.
\end{corollary}

\begin{corollary}\label{proj}
  For a Boolean locale~\(X\),
  the final object of \(\Shv(X)\)
  is projective
  in the sense that
  any epimorphism \(\shf{F}\to{*}\) has a section.
\end{corollary}

\subsection{Linear algebra on a Boolean locale}\label{ss:shv_k}

In this section,
we prove the following structure theorem:

\begin{theorem}\label{b_hla}
  For a Boolean locale~\(X\)
  and a field~\(k\),
  any object of \(\Shv(X;\D(k))\)
  is (noncanonically) equivalent
  to a direct sum
  of objects of the form \(\Sigma^n{k_U}\)
  where \(n\) is an integer
  and \(U\) is an open of~\(X\).
  Here we write \(k_{U}\)
  for the extension by zero
  of the constant sheaf~\(k\) on~\(U\).
\end{theorem}

We first prove the following underived variant;
note that \(\D(k)^{\heartsuit}\) denotes
the \(1\)-category of discrete \(k\)-vector spaces:

\begin{proposition}\label{b_la}
  For a Boolean locale~\(X\)
  and a field~\(k\),
  any object of
  the \(1\)-category of discrete \(k\)-linear sheaves
  \(\Shv(X;\D(k)^{\heartsuit})\)
  is (noncanonically) 
  the direct sum of the sheaves of the form~\(k_U\)
  for \(U\subset X\) an open.
\end{proposition}

We need several lemmas to prove this:

\begin{lemma}\label{c69ee4a6e9}
  Let \(X\) be a locale
  and \(k\) a field.
  Any subobject of \(k_X\in\Shv(X;\D(k)^{\heartsuit})\)
  is of the form~\(k_U\hookrightarrow k_X\)
  for a unique open~\(U\).
\end{lemma}

\begin{proof}
  Let \(\shf{F}\hookrightarrow k_X\) be an arbitrary subobject.
  In the category
  of set-valued sheaves \(\Shv(X;\Cat{Set})\),
  we base change this morphism
  along \(1\colon{*}\to k_X\)
  to obtain a monomorphism to~\(*\).
  It is written as
  \({*}_U\hookrightarrow{*}\)
  for a unique open~\(U\).
  Hence we have a map \({*}_U\to\shf{F}\)
  and
  also get a map \(k_U\to\shf{F}\) by adjunction.
  Therefore,
  we have a morphism of subobjects
  \(k_U\hookrightarrow\shf{F}\) of~\(k_X\)
  that is an isomorphism
  after base change along \(1\colon{*}\to k_X\).
  By \(k^{\times}\)-equivariance,
  this is an isomorphism
  after base change along \(i\colon{*}\to k_X\)
  for \(i\in k^{\times}\).
  This is also an isomorphism
  after base change along \(0\colon{*}\to k_X\)
  as \(k_U\times_{k_X}{*}={*}\).
  Since the map \(\coprod_{i\in k}{*}\to k_X\)
  is an isomorphism,
  \(k_U\hookrightarrow\shf{F}\) is an isomorphism.
  The uniqueness is clear.
\end{proof}

\begin{lemma}\label{b5c46f15e8}
  In the situation of \cref{b_la},
  an epimorphism
  from \(k_X\)
  in \(\Shv(X;\D(k)^{\heartsuit})\)
  is isomorphic to
  the projection
  \(k_X\simeq k_U\oplus k_{X\setminus U}\to k_U\)
  for a unique open~\(U\).
\end{lemma}

\begin{proof}
  This follows from
  \cref{c69ee4a6e9}
  by taking the kernel of a given epimorphism.
\end{proof}

\begin{lemma}\label{0a9344a400}
  In the situation of \cref{b_la},
  let
  \(\shf{F}\to\shf{F}'\)
  be a morphism in \(\Shv(X;\D(k)^{\heartsuit})\).
  It is an epimorphism
  if and only if
  the map \(\shf{F}(X)\to\shf{F}'(X)\)
  is surjective.
\end{lemma}

\begin{proof}
  We first prove the ``if'' direction.
  Then we can show that it is surjective
  on local sections
  by \cref{shv_b},
  since
  for any open~\(U\),
  the map \(\shf{F}(U)\to\shf{F}'(U)\) is a retract of
  \begin{equation*}
    \shf{F}(X)
    \simeq
    \shf{F}(U)\oplus\shf{F}(X\setminus U)
    \to
    \shf{F}'(U)\oplus\shf{F}'(X\setminus U)
    \simeq
    \shf{F}'(X)
  \end{equation*}

  We prove the ``only if'' direction,
  which means that \(k_X\) is projective.
  Let \(\shf{F}\to k_X\) be an epimorphism in \(\Shv(X;\D(k)^{\heartsuit})\).
  In \(\Shv(X;\Cat{Set})\)
  we take a base change of this map along \(1\colon{*}\to k_X\)
  to get an epimorphism \(\shf{F}'\to{*}\).
  By \cref{proj},
  we have a section \({*}\to\shf{F}'\).
  The map \(k_X\to\shf{F}\) 
  corresponding to
  the composite \({*}\to\shf{F}'\to\shf{F}\)
  by adjunction gives us a desired splitting.
\end{proof}

\begin{proof}[Proof of \cref{b_la}]
  Let \(\shf{F}\) be a sheaf.
  First,
  by recursion,
  we construct
  an open~\(U_s\)
  and a map \(a_s\colon k_{U_s}\to\shf{F}\)
  for each section \(s\in\shf{F}(X)\):
  We fix a well ordering on \(\shf{F}(X)\).
  For each \(s\in\shf{F}(X)\),
  let \(b_s\colon\bigoplus_{s'<s}k_{U_s}\to\shf{F}\)
  be the map obtained by \((a_{s'})_{s'<s}\).
  Then
  the composite
  \(k_X\xrightarrow{s}\shf{F}\to\coker(b_s)\)
  is written as
  the composite
  \(k_X\simeq
  k_U\oplus k_{X\setminus U}\twoheadrightarrow k_U\hookrightarrow\coker(b_s)\)
  by \cref{b5c46f15e8}.
  We declare \(U_s\) to be this open~\(U\)
  and \(a_s\) to be
  the composite
  \(k_{U}\to k_{U}\oplus k_{X\setminus U}\simeq k_X\xrightarrow{s}\shf{F}\).

We wish to show that
  the morphism
  \(c\colon\bigoplus_{s\in\shf{F}(X)}
  k_{U_s}\to\shf{F}\)
  induced by \((a_s)_{s\in\shf{F}(X)}\)
  is an isomorphism.
  We first see that \(c\) is a monomorphism
  by induction:
  For each~\(s\),
  we show that
  the map
  \(\bigoplus_{s'\leq s}k_{U_{s'}}\to\shf{F}\)
  induced by~\(a_s\) and~\(b_s\)
  is a monomorphism.
  By the inductive hypothesis,
  \(b_s\) is a monomorphism.
  Since
  \(k_U\xrightarrow{a_s}\shf{F}\to\coker(b_s)\)
  is a monomorphism,
  so is the map induced by~\(a_s\) and~\(b_s\).
  Then we show that
  \(c\) is an epimorphism.
  By \cref{0a9344a400},
  we need to see that
  for \(s\in\shf{F}(X)\),
  the corresponding morphism
  \(k_X\to\shf{F}\)
  factors through~\(c\).
  Since \(b_s\) is a monomorphism,
  we obtain a map
  \begin{equation*}
    k_{X\setminus U_s}
    \simeq
    \ker(k_X\to k_{U_s})
    \to
    \ker(\shf{F}\to\coker(b_s))
    \simeq
    \bigoplus_{s'<s}k_{U_{s'}}.
  \end{equation*}
  Then the coproduct of~\(a_s\) and this map
  gives a desired factorization.
\end{proof}

\begin{corollary}\label{b_hdim}
  For a Boolean locale~\(X\)
  and a field~\(k\),
  the Grothendieck abelian category
  \(\Shv(X;\D(k)^{\heartsuit})\)
  has homological dimension \(\leq0\);
  i.e., every object is injective and projective.
\end{corollary}

\begin{proof}
  It suffices to show that every object is projective.
  By \cref{b_la},
  it suffices to show that
  \(k_U\) for any open~\(U\) is projective.
  Since this is a direct summand of~\(k_X\),
  it is projective by \cref{0a9344a400}.
\end{proof}

\begin{proof}[Proof of \cref{b_hla}]
By \cref{post},
  the \(\infty\)-category \(\Shv(X;\D(k))\)
  is equivalent to the derived \(\infty\)-category of
  its heart.
  Therefore, by \cref{b_hdim},
  any object~\(\shf{F}\) is equivalent to
  \(\bigoplus_n\Sigma^n\pi_n\shf{F}\)
  and
  hence
  \cref{b_la} gives the desired result.
\end{proof}

We also record the following:

\begin{corollary}\label{b_wdim}
  For a Boolean locale~\(X\)
  and a field~\(k\),
  the heart
  \(\Shv(X;\D(k))^{\heartsuit}
  \subset\Shv(X;\D(k))\) is closed under
  tensor product operations.
\end{corollary}

\begin{proof}
  By \cref{b_la},
  it suffices to consider
  the tensor product of finite copies of the unit,
  which is again the unit.
\end{proof}

\section{The smashing spectrum of \texorpdfstring{\(\D(k)\)}{D(k)}-valued sheaves}\label{s:ex}

We first review facts we need from higher topos theory
in \cref{ss:barr}.
We then prove \cref{main,tloc_k} in \cref{ss:u_k,ss:tloc},
respectively.

\subsection{The existence of a Boolean cover}\label{ss:barr}

We recall the following from~\cite[Corollary~A.4.3.2]{LurieSAG}:

\begin{theorem}[Lurie]\label{barr}
  For any hypercomplete \(\infty\)-topos \(\cat{X}\),
  there is a Boolean locale~\(X\)
  and
  a geometric morphism
  \begin{equation*}
    \Shv(X)\to\cat{X}
  \end{equation*}
  such that its inverse image functor
  is conservative.
\end{theorem}

\begin{remark}\label{a2afe375f6}
  In the situation of \cref{barr},
  when \(\cat{X}\)
  is the hypercompletion of a \(0\)-localic \(\infty\)-topos,
  the statement is simple:
  It says that any frame has a frame injection to
  a Boolean frame.
  This theorem was proven by Funayama~\cite{Funayama59}.
  His proof is quite involved because of its generality
  and a simpler proof in this case can be found,
  e.g.,
  in~\cite[Corollary~II.2.6]{Johnstone82}.
When \(\cat{X}\)
  is the hypercompletion of a \(1\)-localic \(\infty\)-topos,
  the statement was proven by Barr~\cite{Barr74}.
\end{remark}

This theorem is useful to us through the following
observation:

\begin{lemma}\label{08765b6080}
  Let \(\cat{Y}\to\cat{X}\) be a map
  of \(\infty\)-toposes
  whose inverse image functor is conservative.
  Then the inverse image functor
  \(\Shv(\cat{X};\cat{C})\to\Shv(\cat{Y};\cat{C})\)
  for a compactly generated \(\infty\)-category~\(\cat{C}\)
  is conservative as well.
\end{lemma}

\begin{proof}
  Let \(\cat{C}_{0}\) denote the full subcategory
  of compact objects.
  The desired result follows
  from the identification
  of
  \(\Shv(\X;\cat{C})\) with
  the \(\infty\)-category
  of the functors \(\cat{C}_{0}^{\op}\to\X\)
  preserving finite limits.
\end{proof}

\subsection{The smashing spectrum of \texorpdfstring{\(\D(k)\)}{D(k)}-valued sheaves}\label{ss:u_k}

In this section,
we prove \cref{main} using \cref{barr},
whose usefulness
can be seen from the following:

\begin{lemma}\label{x91r9p}
  Let \(\cat{C}\to\cat{D}\) be
  a colimit-preserving symmetric monoidal functor
  between presentably symmetric monoidal \(\infty\)-categories\footnote{Note that here we state this in the unstable setting;
    see~\cite{ttg-sm} for the definition of~\(\Sm\).
  }.
  If it is conservative,
  the induced map \(\Sm(\cat{D})\to\Sm(\cat{C})\)
  of locales is an epimorphism.
\end{lemma}

\begin{proof}
  It induces a conservative frame map
  \(F\to G\),
  where \(F\) and~\(G\)
  are \(\Sm(\cat{C})\) and \(\Sm(\cat{D})\),
  respectively.
  We wish to show that it is injective.
  Suppose that \(U\) and \(U'\in F\) are mapped to the same open.
  In that case,
  \(U\wedge U'\) is also mapped to the same open,
  but by conservativity,
  we get \(U=U\wedge U'=U'\).
\end{proof}

We do not directly apply \cref{x91r9p}
in our proof of \cref{main}.
Instead, we use the following observations:

\begin{lemma}\label{25186dcf09}
  Let \(\cat{X}\) be an \(\infty\)-topos and \(k\) a field.
  For
  any \(\E_{1}\)-algebra~\(A\)
  in \(\Shv(\cat{X};\D(k))\),
  there is a (unique) subterminal object~\(U\)
  such that
  \(A\rvert_{\cat{X}_{/U}}\) is zero
  and
  the unit
  \(k\rvert_{\cat{X}_{\setminus U}}\to\pi_0A\rvert_{\cat{X}_{\setminus U}}\)
  is a monomorphism
  in \(\Shv(\cat{X}_{\setminus U};\D(k)^{\heartsuit})\),
  where \(\cat{X}_{\setminus U}\)
  denotes the closed subtopos
  complementary to~\(\cat{X}_{/U}\).
\end{lemma}

\begin{proof}
In \(\Shv(\cat{X};\Cat{Set})\),
  we base change the monomorphism \(1\colon{*}\to\pi_0A\)
  along \(0\colon{*}\to\pi_0A\)
  to get a subterminal object,
  which corresponds to an open~\(U\).
  We wish to show that
  \(U\) satisfies the requirement.
  By the functoriality of this construction,
  it suffices to consider the cases
  \(U={*}\) and \(U=\emptyset\).

  We first assume \(U={*}\).
  This means that
  \({*}\xrightarrow{1}k\to\pi_0A\)
  is zero,
  which implies \(A=0\)
  as \(A\) is an \(\E_1\)-algebra.

  We then assume \(U=\emptyset\)
  and show that
  \(k\to\pi_0A\) is a monomorphism.
  By considering the decomposition
  \(\coprod_{i\in k}{*}\simeq k\),
  it suffices to show that
  the limit of \({*}\xrightarrow{i}\pi_0A\xleftarrow{i'}{*}\)
  is initial
  when \(i\) and~\(i'\) are distinct elements of~\(k\).
  The assumption
  is the case \((i,i')=(0,1)\).
  The other cases are also reduced to this case
  by considering the automorphism of~\(\pi_0A
  \in\Shv(\cat{X};\D(k)^{\heartsuit})\)
  given by \((i-(\X))(i-i')^{-1}\).
\end{proof}

\begin{lemma}\label{x4dpup}
  For a field~\(k\)
  and an \(\infty\)-topos~\(\cat{X}\),
  suppose that
  \(e\colon k\to E\) is an idempotent object
  in \(\Shv(\cat{X};\D(k))\).
  If \(\pi_0e\colon k\to\pi_0E\) is a monomorphism
  in \(\Shv(\cat{X};\D(k)^{\heartsuit})\),
  \(e\) is an equivalence.
\end{lemma}

\begin{proof}
  We first prove this
  under the assumption that \(\cat{X}=\Shv(X)\) holds
  for a Boolean locale~\(X\).
  By \cref{b_hla,b_hdim},
  \(e\) is an inclusion to the direct summand.
  Hence \(k_{X}\) is a retract of~\(E\).
  Since \({\id}_E\otimes e\) is an equivalence,
  so is \({\id}_{k_{X}}\otimes e\simeq e\).

  We consider the general case.
  By \cref{barr},
  there is
  a Boolean locale~\(Y\)
  and
  a geometric morphism \(\Shv(Y)\to\cat{X}^{\hyp}\)
  whose inverse image is conservative.
  We write \(f\) for its composite with \(\cat{X}^{\hyp}\to\cat{X}\).
  Then
  \(f^*e\colon k_Y\to f^*E\)
  is an idempotent object
  and hence an equivalence
  by the argument above.
  Therefore by \cref{08765b6080},
  the composite
  \(k\xrightarrow{e}E\to E^{\hyp}\)
  is an equivalence
  and thus \(e\) is a split monomorphism
  in \(\CAlg(\Shv(\cat{X};\D(k)))\).
  Since \(e\) is also an epimorphism
  in \(\CAlg(\Shv(\cat{X};\D(k)))\),
it must be an equivalence.
\end{proof}

\begin{proof}[Proof of \cref{main}]
  It is clear that
  \(\Sm(\cat{X})\to X\) is an epimorphism of locales.
  It suffices to show that it induces a surjection of frames.
  Let \(E\) be an idempotent algebra
  in \(\Shv(\cat{X};\D(k))\).
  By
  \cref{25186dcf09},
  we have a subterminal object~\(U\)
  such that \(E\rvert_{\cat{X}_{/U}}\) is zero
  and
  \(E\rvert_{\cat{X}_{\setminus U}}\)
  satisfies the assumption of \cref{x4dpup}.
  By \cref{x4dpup},
  as an idempotent algebra,
  \(E\) is equivalent to
  \(\cofib(k_U\to k)\).
\end{proof}

\subsection{Application: Tannaka duality}\label{ss:tloc}

We deduce a refined version of \cref{tloc_k}
from \cref{main}.
First, we explain its namesake:

\begin{remark}\label{xb7jmw}
  Since the work of Lurie~\cite{Lurie05},
  Tannaka duality in algebraic geometry
  means a categorical reconstruction result
  of certain type.
  For example,
  Lurie showed
  in~\cite[Corollary~9.6.0.2]{LurieSAG}
  that the functor
  \begin{equation*}
    \QCoh\colon
    \{\text{quasicompact quasiseparated spectral algebraic spaces}\}^{\op}
    \to
    \CAlg(\Cat{Pr})
  \end{equation*}
  is fully faithful;
  see~\cite[Remark 9.0.0.5]{LurieSAG} for other references.
\end{remark}

\begin{proof}[Proof of \cref{tloc_k}]
  By~\cite{ttg-sm},
  for a presentably symmetric monoidal \(\infty\)-category~\(\cat{V}\),
  we can compose~\(\Sm\) with 
  the functor
  \(\cat{V}\otimes\X
  \colon\CAlg(\Cat{Pr})\to\CAlg_{\cat{V}}(\Cat{Pr})\)
  to obtain the adjunction
  \begin{equation*}
    \begin{tikzcd}[column sep=huge]
      \Cat{Loc}^{\op}\ar[r,shift left,"\Shv(\X;\cat{V})"]&
      \CAlg_{\cat{V}}(\Cat{Pr})\ar[l,shift left,"\Sm"]\rlap.
    \end{tikzcd}
  \end{equation*}
  The desired statement
  is about the case \(\cat{V}=\D(k)\).
  \Cref{main} says that the unit is an equivalence
  so \(\Shv(\X;\D(k))\) is fully faithful.
\end{proof}

Note that
the \(1\)-category
\(\Cat{Loc}\) can be upgraded
to a \((1,2)\)-category,
i.e., a bicategory
whose mapping categories are (essentially) posets.
Our convention is that
\(\Cat{Loc}(Y,X)\) denotes
the full subcategory
of \(\Fun(F,G)\)
spanned by frame morphisms
where \(F\) and~\(G\)
denote the frames of~\(X\) and~\(Y\),
respectively.
We can show that
this is again fully faithful
in the bicategorical sense:

\begin{theorem}\label{tloc_k_2}
  For any field~\(k\) and locales~\(X\) and~\(Y\),
  the functor
  \begin{equation*}
    \Cat{Loc}(Y,X)
    \to
    \Fun_{\D(k)}^{\otimes}
    \bigl(\Shv(X;\D(k)),\Shv(Y;\D(k))\bigr)
  \end{equation*}
  is an equivalence.
\end{theorem}

\begin{proof}
Note that
  a functor of \(\infty\)-categories \(\cat{C}\to\cat{D}\)
  is an equivalence if and only if
  \(\Fun(\Delta^{1},\cat{C})^{\simeq}
  \to
  \Fun(\Delta^{1},\cat{D})^{\simeq}\) is an equivalence.
  By applying this to our situation,
  the desired result follows
  from \cref{tloc_k}
  by considering the situation
  where
  \(X=X\)
  and \(Y\) is the locale corresponding
  to the frame \(\Fun(\Delta^{1},G)\)
  where \(G\) is the frame of~\(Y\).
\end{proof}

\bibliographystyle{plain}
\let\SS\oldSS  \newcommand{\yyyy}[1]{}

\end{document}